\documentclass[11pt,a4paper,twoside]{article}
\usepackage{amsthm,amsmath,amssymb,graphics, amscd}
\usepackage{graphicx}
\usepackage{psfrag}
\newcommand{\field}[1]{\mathbb{#1}}
\newcommand{\R}{\field{R}}

\newcommand{\Hp}{\field{H}}

\newcommand{\Pe}{\field{P}}
\newcommand{\Q}{\field{Q}}

\theoremstyle{plain}
\newtheorem{theorem}{Theorem}[section]
\newtheorem{proposition}[theorem]{Proposition}
\newtheorem{lemma}[theorem]{Lemma}
\newtheorem{claim}[theorem]{Claim}
\newenvironment{rem}{\bf Remark.\rm}{~\hfill~$\diamond$}

\newenvironment{outline}{\noindent\bf Outline of the proof.\rm}{~\hfill~$\Box$\\}
\begin{document}

\begin{centerline}
{\bf {\LARGE An extension of the Masur domain}}
\end{centerline}
\
\\\
\begin{center}
{\sc {Cyril LECUIRE}}\\
\end{center}
\
\\\
\begin{abstract}
The Masur domain is a subset of the space of projective measured geodesic laminations on the boundary of a $3$-manifold $M$. This domain plays an important role in the study of the hyperbolic structures on the interior of $M$. In this paper, we define an extension of the Masur domain and explain that it shares a lot of properties with the Masur domain.
\end{abstract}

\section{Introduction}
\indent
A compression body is the connected sum along the boundary of a ball of $I$-bundles over closed surfaces and solid tori. Among the compression bodies are the handlebodies which are the connected sums along the boundary of solid tori $D^2\times S^1$.  If $M$ is a compression body and if $\partial M$ has negative Euler characteristic then, by Thurston hyperbolization theorem, its interior admits a hyperbolic structure. Namely there are discrete faithful representations $\rho:\pi_1(M)\rightarrow Isom(\Hp^3)$ such that $\Hp^3/\rho(\pi_1(M))$ is homeomorphic to the interior of $M$. If such a representation $\rho$ is geometrically finite, it is said to uniformize $M$.\\
\indent
 In \cite{masur}, H. Masur studied the space of projective measured foliations on the boundary of a handlebody. He described the limit set of the action of the modular group on this space and defined a subset of the space of projective measured foliations on which this action is properly discontinuous. In \cite{chefjeune}, J.-P. Otal defined a similar subset ${\cal O}$ of the space of projective measured geodesic laminations on the boundaries of compression bodies. This set ${\cal O}\subset{\cal PML}(\partial M)$ is called the Masur domain and J.-P. Otal showed that the action of the modular group on ${\cal O}$ is properly discontinuous. He also proved the following : if $int(M)$ is endowed with a convex cocompact hyperbolic metric, then any projective class of measured geodesic laminations lying in ${\cal O}$ is realized by a pleated surface. He also showed that the  injectivity theorem of \cite{thui} applies for such pleated surfaces.\\
\indent
Later it was shown that the projective classes of measured laminations in ${\cal O}$ are an analogous of what Thurston called binding laminations on $I$-bundles over closed surfaces. Namely if we have a sequence of geometrically finite representations\linebreak $\rho_n:\pi_1(M)\rightarrow Isom(\Hp^3)$ uniformizing a compression body and a measured geodesic lamination $\lambda\in{\cal O}$ such that $l_{\rho_n}(\lambda)$ is bounded, then the sequence $(\rho_n)$ contains\linebreak an algebraically converging subsequence. This property has been obtained for various cases in \cite{thuii}, \cite{conti}, \cite{canary}, \cite{ohsh} and the general statement comes from \cite{boches} and \cite{boches1}.\\

\indent
In  this paper, we allow $M$ to be any orientable $3$-manifold with boundary satisfying the following : the Euler characteristic of $\partial M$ is negative and the interior of $M$ admits a complete hyperbolic metric. We will consider the following set :\medskip\\
\indent
${\cal D}(M)=\{\lambda\in{\cal ML}(\partial M)|\,\exists\eta>0$  such that $i(\lambda,\partial E)>\eta$ for any essential annulus or disc $E\subset M\}$.\\

\indent 
First we will link this set ${\cal D}(M)$ with the result of \cite{espoir} and deduce from this that the support of a geodesic measured lamination lying in ${\cal D}(M)$ is also the support of a (in fact many) bending measured geodesic lamination of a representation uniformizing $M$. Using the continuity of the bending measure proved in \cite{kes} and \cite{bodiff}, we will show that ${\cal D}(M)$ is connected. It follows from the ideas of \cite{chefjeune} that the projection of ${\cal D}(M)$ on ${\cal PML}(\partial M)$ contains ${\cal O}$ and we will use this to show that the Masur domain is connected.\\

\indent
After that, we will prove that the set  ${\cal D}(M)$ has the following properties :\medskip\\
\indent If $int(M)$ is endowed with a convex cocompact hyperbolic metric, any measured geodesic lamination lying in ${\cal D}(M)$ is realized by a pleated surface and such a pleated surface satisfies the injectivity theorem of \cite{thui}.\medskip\\
\indent If $\rho_n$ is a sequence of geometrically finite metrics uniformizing $M$ and $\lambda\in{\cal D}(M)$ is a measured geodesic lamination such that $l_{\rho_n}(\lambda)$ is bounded, then the sequence $(\rho_n)$ contains an algebraically converging subsequence.\\

\indent We will also discuss the action of the modular group on ${\cal D}(M)$.\\

\indent
I would like to thank F. Bonahon, I. Kim, K. Ohshika and J.-P. Otal for fruitful discussions and J. Souto who gave me the ideas of Proposition \ref{arc}.

\section{Definitions}
\subsection{Geodesic Laminations}
\indent  Let $S$ be a closed surface endowed with a complete hyperbolic metric; a  {\it geodesic lamination} on $S$ is a compact subset that is the disjoint union of complete embedded geodesics. Using the fact that two complete hyperbolic metrics on $S$ are quasi-isometric, this definition can be made independent of the chosen metric on $S$ (see \cite{conti} for example). A geodesic lamination whose leaves are all closed is called a  {\it multi-curve}. If each half-leaf of a geodesic lamination $L$ is dense in $L $, then $L $ is  {\it minimal}. Such a minimal geodesic lamination is either a simple closed curve or an  {\it irrational lamination}. A leaf $l$ of a geodesic lamination $L $ is  {\it recurrent} if it lies in a minimal geodesic lamination. Any geodesic lamination is the disjoint union of  finitely many minimal laminations and non-recurrent leaves. A leaf is said to be an  {\it isolated} leaf if it is either a non-recurrent leaf or a compact leaf without any leaf spiraling toward it.\\
\indent Let $L$ be a connected geodesic lamination which is not a simple closed curve and let us denote by $\bar S(L)$ the smallest surface with geodesic boundary containing $L$. Inside $\bar S(L)$ there are finitely many closed geodesics (including the components of $\partial\bar S(L)$)  disjoint from $L$ and these closed geodesics do not intersect each other (cf. \cite{espoir}); let us denote by $\partial'\bar S(L)\supset\partial\bar S(L)$ the union of these geodesics. Let us remove from $\bar S(L)$ a small tubular neighbourhood of $\partial'\bar S(L)$ and let $S(L)$ be the resulting surface. We will call $S(L)$ the {\it surface embraced} by the geodesic lamination $L$ and $\partial'\bar S(L)$ the {\it effective boundary} of $S(L)$. If $L$ is a simple closed curve, let us define $S(L)$ to be an annular neighbourhood of $L$ and $\partial'\bar S(L)=L$. If $L$ is not connected, $S(L)$ is the disjoint union of  the surfaces embraced by the connected components of $L$ and $\partial'\bar S(L)=\bigcup_{\{L^i\mbox{ is a component of } L\}}\partial'\bar S(L^i)$.\\

\indent
A  {\it measured geodesic lamination} $\lambda$ is a transverse measure for some geodesic lamination $|\lambda|$: any arc $k\approx [0,1]$ embedded in $S$ transversely to $|\lambda|$, such that\linebreak $\partial k\subset S-\lambda$, is endowed with an  additive measure $d\lambda$ such that :\\
\indent - the support of $d\lambda_{|k}$ is $|\lambda|\cap k$;\\
\indent - if an arc $k$ can be homotoped into $k'$ by a homotopy respecting $|\lambda|$ then\linebreak $\int_k\! d\lambda=\int_{k'} d\lambda$.\\
We will denote by ${\cal ML}(S)$ the space of measured geodesic lamination topologised with the topology of weak$^*$ convergence. We will denote by $|\lambda|$ the support of a measured geodesic lamination $\lambda$.\\
\indent
Let $\gamma$ be a weighted simple closed geodesic with support $|\gamma|$ and weight $w$ and let $\lambda$ be a measured geodesic lamination, the intersection number between $\gamma$ and $\lambda$ is defined by $i(\gamma,\lambda)=w \int_{|\gamma|} d\lambda$. The weighted simple closed curves are dense in ${\cal ML}(S)$ and this intersection number extends continuously to a function\linebreak $i:{\cal ML}(S)\times{\cal ML}(S)\rightarrow\R$ (cf. \cite{bouts}). A measured geodesic lamination $\lambda$ is {\it arational} if for any simple closed curve $i(c,\lambda)=\int_c d\lambda>0$.

\subsection{Real trees}	\label{def}
\indent
An $\R$-tree ${\cal T}$ is a metric space such any two points $x,y$ can be joined by a unique simple arc. Let $G$ be a group acting by isometries on an $\R$-tree ${\cal T}$; the action is  {\it minimal} if there is no proper invariant subtree and  {\it small} if the stabilizer of any non-degenerate arc is virtually abelian.\\
\indent
A $G$-equivariant map $\phi$ between two $\R$-trees ${\cal T}$ and ${\cal T}'$ is a {\it morphism} if and only if every point $p\in{\cal T}$ lies in a non-degenerate segment $[a,b]$ (but $p$ may be a vertex of $[a,b]$) such that the restriction $\phi_{|[a,b]}$ is an isometry. The point $p$ is a {\it branching point} if there is no segment $[a,b]$ such that $\phi_{|[a,b]}$ is an isometry and that $p\in]a,b[$.\\
\indent
Let $S$ be a connected hyperbolic surface and let $q:\Hp^2\rightarrow S$ be the covering projection. Let $L\subset S$ be a geodesic lamination and let $\pi_1(S)\curvearrowright{\cal T}$ be a minimal action of $\pi_1(S)$ on an  $\R$-tree ${\cal T}$; $L$ is realized in ${\cal T}$ if there is a continuous equivariant map $\Hp^2\rightarrow {\cal T}$ whose restriction to any lift of a leaf of $L$ is injective.\\
\indent
Let $\lambda\in{\cal ML}(S)$ be a measured geodesic lamination; following \cite{chefm}, we will define the dual tree of $\lambda$. Consider the following metric space $pre{\cal T}_{\lambda}$ : the points of $pre{\cal T}_{\lambda}$ are the complementary regions of $q^{-1}(\lambda)$ in $\Hp^2$, where $q:\Hp^2\rightarrow S$ is the covering projection and the distance $d:{\cal T}_{\lambda}\times {\cal T}_{\lambda}\rightarrow \R$ is defined as follows. Let $R_0$ and $R_1$ be two complementary regions and choose a geodesic segment $k\subset\Hp^2$ whose vertices lie in $R_0$ and $R_1$; we set $d(R_0,R_1)$ to be the $q^{-1}(\lambda)$-measure of $k$. Then, there is a unique (up to isometry) $\R$-tree ${\cal T}_\lambda$ and an isometric embedding $e:pre{\cal T}_{\lambda}\rightarrow {\cal T}_{\lambda}$ such that {\bf any point} of ${\cal T}_{\lambda}$ lies in a segment with endpoints in $e(pre{\cal T}_{\lambda})$ (cf. \cite{euh}). The covering transformations yield an isometric action of $\pi_1(M)$ on ${\cal T}_{\lambda}$; if $\delta_{\lambda}(c)$ is the distance of translation of an isometry of ${\cal T}_{\lambda}$ corresponding to a simple closed curve $c$, we have $\delta_{\lambda}(c)=i(c,\lambda)$. This construction yields a natural projection $\Hp^2-q^{-1}(\lambda)\rightarrow {\cal T}_{\lambda}$. If $\lambda$ does not have closed leaves, this projection extends continuously to a map $\pi_{\lambda}:\Hp^2\rightarrow {\cal T}_{\lambda}$. Otherwise, replacing closed leaves of $\lambda$ by foliated annuli endowed with uniform transverse measures, we get also a continuous map $\pi_{\lambda}:\Hp^2\rightarrow {\cal T}_{\lambda}$ (cf. \cite{conti}).

\subsection{Train tracks}
A  {\it train track } $\tau$ in $S$ is the union of finitely many "rectangles" $b_i$ called the  {\it branches} and satisfying:
\begin{description}
\item -  any branch $b_i$ is an imbedded rectangle $[0,1]\times [0,1]$ such that the preimage of the double points is a segment of $\{0\}\times [0,1]$ and a segment of $\{1\}\times [0,1]$;
\item -  the intersection of two different branches is either empty or a non-degenerate segment lying in the vertical sides $\{0\}\times [0,1]$ and $\{1\}\times [0,1]$;
\item -  any connected component of the union of the vertical sides is a simple arc embedded in $\partial_{\chi<0} M$.
\end{description}
\indent A connected component of the union of the vertical sides is a  {\it switch}. In each branch the segments $\{p\}\times [0,1]$ are the {\it ties} and the segments $[0,1]\times\{p\}$ are the {\it rails}.\medskip\\
\indent
A geodesic lamination $L $ is  {\it carried by a train track} $\tau$ when:
\begin{description}
\item -  $L $ lies in $\tau$;
\item -  for each branch $b_i$ of $\tau$,  $L \cap b_i$ is not empty, lies in the image of $[0,1]\times ]0,1[$ and each leaf of $L $ is transverse to the ties.
\end{description}
Notice that, in some papers, a geodesic lamination satisfying the above is said to be ``minimally carried'' by $\tau$.\\
\indent A measured geodesic lamination $\lambda$ is carried by a train track $\tau$ if its support $|\lambda|$ is carried by $\tau$.\\
\indent Let $S$ be a hyperbolic surface, let $\tau\subset S$ be a train track and let $\pi_1(M)\curvearrowright{\cal T}$ be a minimal action of $\pi_1(M)$ on an  $\R$-tree ${\cal T}$. Let $q^{-1}(\tau)\subset\Hp^2$ be the preimage of $\tau$ under the covering projection; a {\it weak realization} of $\tau$ in ${\cal T}$, is a $\pi_1(M)$-equivariant continuous map $\pi:q^{-1}(\tau)\rightarrow{\cal T}$ such that $\pi$ is constant on the ties of $q^{-1}(\tau)$, monotone and not constant on the rails and that the images of two adjacents branches lying on opposite sides of the same switch have disjoint interiors.

\subsection{$3$-manifolds}
Let $M$ be a $3$-manifold, $M$ is {\it irreducible} if any sphere embedded in $M$ bounds a ball. We will say that $M$ is a  {\it hyperbolic manifold} if its interior can be endowed with a complete hyperbolic metric. Let $\Sigma$ be a subsurface of $\partial M$; an  {\it essential disc} in $(M,\Sigma)$ is a disc $D$ properly embedded in $(M,\Sigma)$ that can not be mapped to $\partial M$ by a homotopy fixing $\partial D$. The simple closed curve $\partial D$ is a {\it meridian} curve. The manifold $M$ is {\it boundary irreducible} if there is no essential disc in $(M,\partial M)$. An  {\it essential annulus} in $(M,\Sigma)$ is an incompressible annulus $A$ properly embedded in $(M,\Sigma)$ which can not be mapped to $\partial M$ by a homotopy fixing $\partial A$. Let $A$ be an essential annulus in $M$; if one component of $\partial A$ lies in a toric component of $\partial M$ we will call the other component of $\partial A$ a  {\it parabolic curve}.\\
\indent
Let $m\subset\partial M$ be a simple closed curve; a simple arc $k$ such that $k\cap m=\partial k$ is an $m$-wave if there is an arc $k'\subset m$ such that $k'\cup k$ bounds an essential disc. A leaf $\tilde l$ of a geodesic lamination $\tilde L\subset\partial\tilde M$ is {\it homoclinic} if it contains two sequences of points $(x_n)$ and $(y_n)$ such that the distance between the points $x_n$ and $y_n$ measured on $\tilde l$ goes to $\infty$ whereas their distance measured in $\tilde M$ is bounded. A leaf $l$ of a geodesic lamination $L\subset\partial M$ is {\it homoclinic} if a (any) lift of $l$ to $\partial\tilde M$ is a homoclinic leaf. Notice that, with this definition, a meridian or a leaf spiralling around a meridian is homoclinic.\\

\indent
Let $\rho:\pi_1(M)\rightarrow Isom(\Hp^3)$ be a faithful discrete representation such that\linebreak $\Hp^3/\rho(\pi_1(M))$ is homeomorphic to the interior of $M$. Let $L_\rho\subset S^2=\partial\overline{\Hp}^3$ be the limit set of $\rho(\pi_1(M))$, let $C(\rho)\subset\Hp^3$ be the convex hull of $L_\rho$ and let $C(\rho)^{ep}$ be the intersection of $C(\rho)$ with the preimage of the thick part of $\Hp^3/\rho(\pi_1(M))$. The quotient $N(\rho)$ of $C(\rho)$ by $\rho(\pi_1(M))$ is the convex core of $\rho$ and $\rho$ is said to be {\it geometrically finite} if $N(\rho)$ has finite volume. A geometrically finite representation $\rho:\pi_1(M)\rightarrow Isom(\Hp^3)$ such that $\Hp^3/\rho(\pi_1(M))$ is homeomorphic to the interior of $M$ is said to {\it uniformize} $M$. If $\rho$  uniformize $M$, there is a natural homeomorphism (defined up to homotopy) $h:\tilde M\rightarrow C(\rho)^{ep}$ coming from the retraction map $S^2-L_{\rho}\rightarrow C(\rho)^{ep}$. Let us choose a geometrically finite representation $\rho$ with only rank $2$ maximal parabolic subgroups (namely the maximal subgroups of $\rho(\pi_1(M))$ containing only parabolic isometries have rank $2$). We will define the compactification $\overline{\tilde M}$ of $\tilde M$ as the closure of $h(\tilde M)=C(\rho)^{ep}$ in the usual unit ball compactification of $\Hp^3$. This compactification does not depend on the choice of the representation $\rho$ (see \cite[section 2.1]{espoir}). We will call this compactification the Floyd-Gromov compactification of $\tilde M$.\\
\indent
Let $\tilde l_+\subset\partial\tilde M$ be a half-geodesic and let $\Bar{\tilde l}_+$ be its closure in $\overline{\tilde M}$; we will say that $\tilde l_+$ {\it has a well defined endpoint} if $\Bar{\tilde l}_+-\tilde l_+$ contains one point. We will say that a geodesic $\tilde l\subset\partial\tilde M$ has two well defined endpoints if $\tilde l$ contains two disjoints half geodesics each having a well defined endpoint. Two distincts leaves $\tilde l_1$ and $\tilde l_2$ of a geodesic lamination $\tilde L\subset\partial\tilde M$ will be said to be {\it biasymptotic} if they both have two well defined endpoints in $\overline{\tilde M}$ and if the endpoints of $\tilde l_1$ are the same as the endpoints of $\tilde l_2$. A geodesic lamination $A\subset\partial M$ is {\it annular} if the preimage of $A$ in $\partial\tilde M$ contains a pair of biasymptotic leaves.\\

\subsection{Pleated surfaces}
Let $\rho:\pi_1(M)\rightarrow Isom(\Hp^3)$ be a discrete faithful representation and let $N=\Hp^3/\rho(\pi_1(M)$. A  {\it pleated surface} in $N$ is a map $f:S\rightarrow N$  from a surface $S$ to $N$ with the following properties :
\begin{description}
\item - the path metric obtained by pulling back the hyperbolic metric of $N$ by $f$ is a hyperbolic metric $s$ on $S$;
\item - every point in $S$ liess in the interior of some $s$-geodesic arc that is mapped to a geodesic arc in $N$;
\end{description}
\indent
The  {\it pleating locus} of a pleated surface is the set of points of $S$ where the map fails to be a local isometry. The pleating locus of a pleated map is a geodesic lamination (cf. \cite{notes}).\\
\indent
Let $\rho:\pi_1(M)\rightarrow Isom(\Hp^3)$ be a discrete faithful representation such that there is a homeomorphism $h:int(M)\rightarrow N=\Hp^3/\rho(\pi_1(M))$ and let $S\subset M$ be a properly embedded surface homeomorphic and homotopic to $\partial M$. A measured geodesic lamination $\lambda\in{\cal ML}(\partial M)$ is realized by a pleated surface in $N$ if there is a pleated surface $f:S\rightarrow N$ homotopic to $h_{|S}$ such that the restriction of $f$ to the support of $\lambda$ is an isometry. 

\subsection{Masur domain}
Let $M$ be a compression body; its boundary has a unique compressible component, the {\it exterior boundary} that we will denote by  $\partial_e M$. Let ${\cal PML}(\partial_e M)$ be the space of projective measured geodesic laminations on $\partial_e M$ and let ${\cal M}'$ be the closure in ${\cal PML}(\partial_e M)$ of the set of projective classes of weighted meridians. The compression body $M$ is said to be a {\it small compression body} if it is the connected sum along the boundary of two $I$-bundles over closed surfaces or of a solid torus and of an $I$-bundle over a closed surface  and is said to be a {\it large compression body} otherwise.  When $M$ is a large compression body, the Masur domain is defined as follows :\\
$${\cal O}=\{\lambda\in{\cal PML}(\partial_e M)|i(\lambda,\mu)>0 \mbox{ for any }\mu\in{\cal M'}\}.$$
\indent When $M$ is a small compression body, the definition is the following one \medskip\\
\indent ${\cal O}=\{\lambda\in{\cal PML}(\partial_e M)|\, i(\lambda,\nu)>0$ for any $\nu\in{\cal PML}(\partial_e BC)$ such that there is $\mu\in{\cal M'}$ with $ i(\mu,\nu)=0\}.$\medskip\\
We will denote by $\hat {\cal O}\subset{\cal ML}(\partial M)$ the set of measured geodesic laminations whose projective class lies in ${\cal O}$.\\
\indent Let $M$ be an orientable hyperbolic $3$-manifold such that $\partial M$ has negative Euler characteristic.  We will say that a measured geodesic lamination $\lambda\in{\cal ML}(\partial M)$ is {\it doubly incompressible} if and only if :\\

\indent
- $\exists\eta>0$  such that $i(\lambda,\partial E)\geq\eta$ for any essential annulus or disc $E$.\\

We will denote by ${\cal D}(M)\subset{\cal ML}(\partial M)$ the set of doubly incompressible measured geodesic laminations.\\
\indent
Doubly incompressible multi-curve were first introduced by W. Thurston in \cite{thuiii} and we have the following equivalence :  $(\partial M, |\gamma|,\subset)$ is doubly incompressible (in the sense of \cite{thuiii}) if and only if there is a weighted multi-curve $\gamma\subset {\cal ML}(\partial M)$ with support $|\gamma|$ satisfying the condition above except in the following situation (in which $\gamma$ lies in ${\cal D}(M)$ but $(\partial M,|\gamma|,\subset)$ is not doubly incompressible in Thurston's sense):\\
\indent - $(-)$ there is a homeomorphism between $M$ and an $I$-bundle over a pair of pants $P$ such that $|\gamma|$ is mapped to a section of the bundle over $\partial P$.\\

\indent
The set ${\cal D}(M)$ of doubly incompressible measured geodesic laminations is the extension of Masur domain we will study in this paper.

\section{Relations between ${\cal O}(M)$, ${\cal D}(M)$ and ${\cal P}(M)$}
When a statement deals with the Masur domain, it means that we have assumed that $M$ is a compression body.

\begin{lemma}	\label{cont}
The set $\hat{\cal O}$ is a subset of ${\cal D}(M)$.
\end{lemma}

\begin{proof}
Let $\lambda\not\in{\cal D}( M)$ be a measured geodesic lamination. We will show, using the following lemma of \cite{chefjeune}, that $\lambda\not\in\hat{\cal O}$.

\begin{lemma}[\cite{chefjeune}]		\label{re}
Let $E$ be an essential annulus in a large compression body $M$; then there is a projective measured geodesic lamination $\mu\in{\cal M}'$ with support lying in $\partial E$.
\end{lemma}

\begin{proof}
Since \cite{chefjeune} is not published, we will write the details of the proof. The boundary of $\partial M$ has only one compressible component $\partial _e M$ called the exterior boundary. Let us choose a complete hyperbolic metric on $\partial_e M$.\\

\begin{claim}
Let $c\subset\partial_e M$ be a simple closed curve that is disjoint from one non separating meridian or from two separating meridians; then there is a projective measured geodesic lamination $\mu\in{\cal M}'$ whose support is $c$.
\end{claim}

\begin{proof}
Let us first consider that there is a non separating meridian $m$ disjoint from $c$. Let $D$ be an essential disc bounded by $c$. Since $c$ does not separate $\partial M$, there is a sequence of simple closed curves $(c_i)$ that approximates $c$ and intersect $m$ in one point, namely the sequence $(c_i)$ converges to $c$ in ${\cal PML}(\partial M)$. Consider a small neighbourhood ${\cal V}_i$ of $D\cup c_i$ in $M$. The closure of $\partial{\cal V}_i-\partial M$ is an essential disc $D_i$ and the sequence $(\partial D_i)$ converges to $c$ in ${\cal PML}(\partial M)$.\\
\indent
Let us now assume that there are two disjoint separating meridians $m_1$ and $m_2$ which do not intersect $c$. Let $D_1$ and $D_2$ be two essential discs bounded by $m_1$ and $m_2$ respectively. Let $N$ be the closure of the connected component of $M-(D_1\cup D_2)$ whose boundary contains $c$. If $N$ intersects $D_1$ and $D_2$, we can approximate $c$ by a sequence of arcs $k_i$ joining $m_1$ to $m_2$. Let ${\cal V}_i$ be a small neighbourhood of $D_1\cup k\cup D_2$. The closure of  $\partial{\cal V}_i-\partial M$ is an essential disc $\Delta_i$ and the sequence $(\partial \Delta_i)$ converges to $c$ in ${\cal PML}(\partial M)$.\\
\indent If $N$ intersects only one disc $D_1$ or $D_2$, by considering an arc in $\partial M-N$ joining $D_1$ and $D_2$, we can construct an essential disc $D_3$ such that one component of $M-(D_1\cup D_3)$ or of $M-(D_2\cup D_3)$ contains $c$ and intersects $D_1$ and $D_3$ or $D_2$ and $D_3$. Thus we are in the previous case and can conclude as above.
\end{proof}

\indent
To prove Lemma \ref{re}, it remains to consider the case where there is at most one meridian disjoint from $E$ and this meridian separates $M$.\\
\indent
Let us assume that the two components of $\partial E$ are not homotopic in $\partial M$. Since $M$ is a large compression body, $E$ intersects a meridian $c$. Let us choose an orientation for $E$ and let $\psi:M\rightarrow M$ be the Dehn twist along $E$. The curve $\psi^n(c)$ is a meridian. The restriction of $\psi_n$ to $\partial M$ is a Dehn twist along $\partial E$. It follows that the sequence $(\psi^n(c))$ tends to a projective measured geodesic lamination $\mu\in{\cal M}'$ with $|\mu|\subset\partial E$.\\
\indent
Consider now that there is an annulus $E'\partial M$ with $\partial E'=\partial E$. By cutting $M$ along an essential disc disjoint from $E$ (if there is one, we can assume that $E$ intersects any essential disc in $M$. Since $M$ is atoroidal and $E\cup E'$ bounds a solid torus $T\subset M$. Furthermore each component of $\partial E'$ represent an element in $\pi_1(M)$ which is divisible. It follows that when $M$ is described as the connected sum along the boundary of tori and $I$-bundle over closed surfaces, $T$ does not go through an $I$-bundle over a closed surface. Since $T$ intersects any essential disc, we get that $M$ is a solid torus. Recalling that we may have cut $M$ along an essential disc, we conclude that $M$ was originally the connected sum along the boundary of a solid torus and an $I$-bundle over a closed surface. This contradicts our asumption that $M$ is a large compression body.\\
\end{proof}

\begin{rem}
If $E$ is an essential annulus in a small compression body, either $\partial E$ intersects a meridian and from the above a measured geodesic sublamination of $\partial E$ lies in ${\cal M}'$, or $\partial E$ is disjoint from the meridian. 
\end{rem}\\

 Let $\lambda$ be a measured geodesic lamination such that $\lambda\not\in {\cal D}(M)$. Then there is a sequence of essential discs or annuli  $E_n\subset M$ such that $i(\lambda,\partial E_n)\longrightarrow 0$. We will show that $\lambda\not\in\hat{\cal O}$.\\
\indent
We will first assume that $M$ is a large compression body. By Lemma \ref{re}, there is a sequence of multi-curves $(e_n)$ such that $e_n\subset\partial E_n$ and  that $e_n\in{\cal M}'$. Let $\varepsilon>0$ and let $\varepsilon e_n$ be the weighted multi-curve obtained by endowing each leaf of $e_n$ with a Dirac mass with weight $\varepsilon$. Up to extracting a subsequence, there is a sequence $(\varepsilon_n)$ converging to $0$ such that the sequence $(\varepsilon_n e_n)$  converges to some measured geodesic lamination $\alpha$. Since $\varepsilon_n e_n\in{\cal M}'$ for any $n$, then $\alpha\in{\cal M}'$. Since $\varepsilon_n\longrightarrow 0$, we have $i(\lambda,\alpha)=0$ hence $\lambda\not\in\hat{\cal O}$.\\
\indent
Let us now assume that $M$ is a small compression body. By the proof of Lemma \ref{re}, for each $n$, either $E_n$ is disjoint from an essential meridian or a connected component of $\partial E_n$ is the support of an element of ${\cal M'}$. Especially, for any $n$, there is $\mu_n\in{\cal M}'$ with $i(\mu_n,e_n)=0$. Furthermore, we can choose the $\mu_n$ such that a subsequence of $(\mu_n)$ converges in ${\cal ML}(\partial M)$ to a measured geodesic lamination $\mu\in{\cal M}'$. We get then $i(\alpha,\mu)=0$ and $i(\alpha,\lambda)=0$ hence $\lambda\not\in\hat{\cal O}$. Thus we have shown that if $\lambda\not\in {\cal D}(M)$, then $\lambda\not\in\hat{\cal O}$.
\end{proof}

The opposite is not true but we have the following :

\begin{lemma}	\label{arra}
Let $\lambda\in{\cal D}(M)$ be an arational measured geodesic lamination; then $\lambda$  lies in $\hat{\cal O}$.
\end{lemma}

\begin{proof}
Let us assume the contrary; if $M$ is a large compression body, there is  $\mu\in{\cal M}'$ such that $i(\mu,\lambda)= 0$. It follows from the assumption that $\lambda$ is arational that $\lambda$ and $\mu$ share the same support $|\mu|$. Since $\mu\in{\cal M}'$, there is a sequence of meridians $c_n\subset \partial M$ and a sequence $\varepsilon_n\longrightarrow 0$ such that $\varepsilon_n c_n$ converges to $\mu$ in the topology of ${\cal ML}(\partial M)$. Up to extracting a subsequence, $(c_n)$ converges in the Hausdorff topology to a geodesic lamination $L$ and we have $|\mu|\subset L$. By Casson's criterion (cf. \cite{chefjeune}, \cite[Theorem B.1]{espoir} or \cite{fini}), $L$ contains a homoclinic leaf $l$. Since $|\mu|\subset L$ is the support of $\lambda$, $l$ does not intersect $\lambda$ transversely. This contradicts Lemma \ref{supcond} below.\\
\indent
If $M$ is a small compression body, $\partial M$ contains a unique meridian $c$. Let us assume that $\lambda\not\in\hat{\cal O}$; then there is $\mu\in{\cal ML}(\partial_e M)$ such that $i(c,\mu)=0$ and $i(\lambda,\mu)=0$. Since $\lambda$ is arational and $i(\lambda,\mu)=0$, $\mu$ is also arational. This contradicts the fact that\linebreak $i(c,\mu)=0$.
\end{proof}

\indent
In \cite{espoir} (see also \cite{fini}), one studied the subset ${\cal P}(M)$ of ${\cal ML}(\partial M)$ defined as follows. Let $\lambda\in{\cal ML}(\partial M)$ be a measured geodesic lamination; then $\lambda\in{\cal P}(M)$ if and only if :
\begin{description}
\item -  $a)$ no closed leaf of $\lambda$ has a weight greater than $\pi$;
\item -  $b)$ $\exists\eta>0$ such that, for any essential annulus $E$, $i(\partial E,\lambda)\geq\eta$;
\item -  $c)$ $i(\lambda,\partial D)>2\pi$ for any essential disc $D$.
\end{description}
Let $\rho:\pi_1(M)\rightarrow Isom(\Hp^3)$ be a geometrically finite representation uniformizing $M$ and let $h$ be an isotopy class of homeomorphisms $M\rightarrow N(\rho)^{ep}$ homotopic to the identity; we will denote by ${\cal GF}(M)$ the set of such pairs $(\rho,h)$. There is a well defined map\linebreak $b:{\cal GF}(M)\rightarrow {\cal ML}(\partial M)$ which to a pair $(\rho,h)$ associates the preimage under $h$ of the bending measured geodesic lamination of $N(\rho)$, let us call this map the bending map. It is shown in \cite{meister1} and \cite{espoir} that ${\cal P}(M)$ is the image of $b$.\\

\indent In \cite{espoir}, it was proved that a measured geodesic lamination lying in ${\cal P}(M)$ intersects transversely all the homoclinic leaves and all the annular laminations. In order to get the same property for the laminations lying in ${\cal D}(M)$, we will discuss the relationships between ${\cal P}(M)$ and ${\cal D}(M)$.\\

\indent We clearly have ${\cal P}(\partial M)\subset {\cal D(M)}$, conversely, we have : 

\begin{lemma}	\label{dedan}
Let $\lambda\in{\cal D(M)}$ be a measured geodesic lamination not satisfying the condition $(-)$, then there is a measured geodesic lamination $\alpha\in{\cal P}(M)$ with the same support as $\lambda$. 
\end{lemma}
\begin{proof}
\noindent Since $\lambda\in{\cal D(M)}$, $\exists \eta>0$ such that $i(\partial E,\lambda)>\eta$ for any essential annulus or disc $E$. Let $\frac{2\pi}{\eta}\lambda$ be the measured geodesic lamination obtained by multiplying the measure $\lambda$ by $\frac{2\pi}{\eta}$; then $\frac{2\pi}{\eta}\lambda$ satisfies the properties $b)$ and $c)$ above. Let $\lambda^{(p)}$ be the union of the leaves of $\frac{2\pi}{\eta}\lambda$ with a weight greater than $\pi$ and let $\alpha$ be the measured geodesic lamination obtained from $\frac{2\pi}{\eta}\lambda$ by decreasing the weight of the leaves of $\lambda^{(p)}$ to $\pi$. This measured geodesic lamination $\alpha$ satisfies $a)$ and $b)$, let us show that it satisfies also $c)$. \\
\indent  Let $D\subset M$ be an essential disc; then $i(\frac{2\pi}{\eta}\lambda,\partial D)>2\pi$. If $\partial D$ does not intersect $\lambda^{(p)}$ transversely, then $i(\alpha,\partial D)=i(\frac{2\pi}{\eta}\lambda,\partial D)>2\pi$.\\
\indent If  $\partial D$  intersects $\lambda^{(p)}$ in one point $x$, let $c $ be the leaf of $\lambda^{(p)}$ containing $x$. Let ${\cal V}$ be a small neighbourhood of $c \cup D$; ${\cal V}$ is a solid torus. Let $D'$ be the closure of $\partial{\cal V}-\partial M$, $D'$ is a disc properly embedded in $M$ which does not intersect $\lambda^{(p)}$. Hence we have $i(\partial D',\alpha)= i(\partial D',\frac{2\pi}{\eta}\lambda)$. If $D'$ is not an essential disc, then $\partial D'$ bounds a disc\linebreak $D''\subset\partial M$. Since $M$ is irreducible, $D'\cup D''$ bounds a ball $B\subset M$ and $M=B\cup{\cal V}$ is a solid torus. By assumption, $M$ is not a solid torus hence $D'$ is an essential disc and $i(\partial D',\frac{2\pi}{\eta}\lambda)>2\pi$. Since $i(\partial D',\alpha)\leq 2(i(\partial D,\alpha)-\pi)$, we have\linebreak $i(\partial D,\alpha)\geq \frac{i(\partial D',\alpha)}{2}+\pi=\frac{i(\partial D',\frac{2\pi}{\eta}\lambda)}{2}+\pi>2\pi$.\\
\indent If $\partial D$ intersects $\lambda^{(p)}$ in two points $x$ and $y$, we have\linebreak $i(\alpha,\partial D)=2\pi+i(\frac{2\pi}{\eta}\lambda-\lambda^{(p)},\partial D)$. Hence we just have to show that\linebreak  $i(\lambda-\lambda^{(p)},\partial D)>0$. Assuming the contrary, we have $\lambda\cap\partial D=\{x,y\}$. If $x$ and $y$ lie in two distincts leaves $c\subset|\lambda|$ and $d\subset|\lambda| $, let ${\cal V}$ be a small neighbourhood of $c \cup d\cup D$; ${\cal V}$ is an $I$-bundle over a pair of pants. The closure of $\partial{\cal V}-\partial M$ is an annulus with boundary not intersecting $|\lambda|$. By condition $b)$, this annulus is not essential. It follows that $M$  is an $I$-bundle over a pair of pants $P$ and that $|\lambda|$ lies in a section of the bundle over $\partial P$. This contradicts our assumptions hence $x$ and $y$ lie in the same leaf $c $ of $\lambda^{(p)}$.\\
\indent Let ${\cal V}$ be a small neighbourhood of $c \cup D$; it is again an $I$-bundle over a pair of pants. If the tangents vectors $\frac{dc}{dt}_{|x}$ and $\frac{dc}{dt}_{|y}$ do not point to the same side of $\partial D$, the closure of $\partial{\cal V}-\partial M$ is the union of two annuli with boundaries not intersecting $\lambda$. This yields the same contradiction as above.\\
\indent Next let us consider the case where $\frac{dc}{dt}_{|x}$ and  $\frac{dc}{dt}_{|y}$  point to the same side of $\partial D$. Let $k$ be a connected component of $c -\{x,y\}$ and let ${\cal V}'$ be a small neighbourhood of $k\cup D$; the closure of $\partial{\cal V}'-\partial M$ is an essential disc $D'$. Replacing $D$ by $D'$, we are in the situation of the previous paragraph and get the same contradiction.\\
\indent If $\partial D$ and $\lambda^{(p)}$ intersect each other in more than $2$ points, $i(\lambda',\partial D)\geq 3\pi$.
\end{proof}

Combining Lemma \ref{dedan} and results of \cite{espoir} (see also \cite{fini}) we get the following :

\begin{lemma}	\label{supcond}
A measured geodesic lamination $\lambda\in{\cal D}(M)$ not satisfying the condition $(-)$  has the following property :
\begin{description}
\item - $\lambda$ intersects transversely any annular lamination and any geodesic lamination containing a homoclinic leaf.
\end{description}
\end{lemma}

\begin{rem}
Let us add a few comments about the case where $\lambda$ satisfies the condition $(-)$. Any homoclinic leaf $l$ intersects $\lambda$ at least once. If an annular geodesic lamination $A$ does not intersect $\lambda$ transversely, then $A$ contains two disjoint half-leaves both spiraling in the same direction toward the same leaf of $\lambda$. This can not happen for a Hausdorff limit of multi-curves. Therefore $\lambda$ has the property above if we consider only annular laminations that are Hausdorff limits of multi-curves.
\end{rem}

\section{Topological properties of ${\cal D}(M)$}

\begin{lemma}
The set ${\cal D}(M)$ is an open set.
\end{lemma}

\begin{proof}
Let us assume the contrary. Then there are $\lambda\in{\cal D}(M)$ and a sequence of measured geodesic laminations $\lambda_n\not\in{\cal D}(M)$  converging to $\lambda$. Therefore there is a sequence of essential discs or annuli $E_n$ such that $i(\lambda_n,\partial E_n)\longrightarrow 0$. Let us extract a subsequence such that $\partial E_n$ converge in the Hausdorff topology to a geodesic lamination $A$. Then $A$ does not intersect $\lambda$ transversely. By \cite{espoir} (see also \cite{fini}) either $A$ contains a homoclinic leaf (\cite[Theorem B1]{espoir}) or $A$ is annular (\cite[Lemma C2]{espoir}), both contradicting Lemma \ref{supcond}.
\end{proof}

\indent
A train track $\tau$ carrying a measured geodesic lamination is {\it complete} if it is not a subtrack of a train track carrying a measured geodesic lamination (cf. \cite{train}).\\
\indent
Any measured geodesic lamination $\lambda$ is carried by some (maybe many) complete train track $\tau$. The weight system on a complete train track gives rise to a coordinate system for a simplex of the piecewise linear manifold ${\cal ML}(\partial M)$. The {\it rational depth} of a measured geodesic lamination $\lambda$ is the dimension of the rational vector space of linear functions with rational coefficients (from the simplex previously defined to $\R$) vanishing on the coordinates of $\lambda$. Let us denote by ${\cal I}(\partial M)$ the set of measured geodesic laminations with rational depth equal to $0$ or $1$. If a measured geodesic lamination $\lambda$ lies in ${\cal I}$, either $\lambda$ is arational or there is a closed leaf $c$ of $\lambda$ such that $\lambda$ is arational in $\partial M-c$ (cf. \cite[Proposition 9.5.12]{notes}). By Lemma \ref{dedan} and \cite[Lemma 2.5]{proper}, the proof of Lemma \ref{arra} holds also in the second case, namely if $\lambda\in{\cal D}(M)$ and if there is  closed leaf $c$ of $\lambda$ such that $\lambda$ is arational in $\partial M-c$, then $\lambda\in\hat{\cal O}$. The set ${\cal I}$ is a dense open subset of ${\cal ML}(\partial M)$ (cf. \cite[chap 9]{notes}).

\begin{proposition}	\label{arc}
The sets ${\cal D}(M)$ and $\hat{\cal O}$ are pathwise connected.
\end{proposition}

\begin{proof}
\indent Let $\lambda_1,\lambda_2\in \hat{\cal O}$. By \cite{interval}, the arational measured geodesic laminations are dense in ${\cal ML}(\partial M)$. Since $\hat{\cal O}$ is open, there are two arational measured geodesic laminations $\alpha_1$ and $\alpha_2\in\hat{\cal O}$ such that $\lambda_j$ is connected to $\alpha_j$ by a path $k_j\subset\hat{\cal O}$. \\
\indent Since  $\alpha_j\in\hat{\cal O}\subset{\cal D}(M)$ there is $\eta>0$ such that $i(\alpha_j,\partial E)>\eta$ for any essential disc or annulus $E\subset M$. Since $\alpha_i$ is arational, it has no closed leaf and, by the proof of Lemma \ref{dedan}, we have $\frac{2\pi}{\eta}\alpha_j\in{\cal P}(M)$. Let ${\cal CC}(M)\subset{\cal GF}(M)$ be the set of  hyperbolic metrics uniformizing $M$ and having only rank $2$ cusps. By results of Ahlfors-Bers (\cite{bers}), ${\cal CC}(M)$ is homeomorphic to the cartesian product of the Teichm\"uller spaces of the connected components of  $\partial_{\chi<0} M$, indeed ${\cal CC}(M)$ is pathwise connected. Let ${\cal P}_{nc}(M)$ be the set of measured geodesic laminations lying in ${\cal P}(M)$ and having no closed leaves with weight $\pi$. By \cite{espoir} (see also \cite{fini}) ${\cal P}_{nc}(M)$ is the image of ${\cal CC}(M)$ by the bending map. By \cite{kes} and \cite{bodiff}, the bending map is continuous on ${\cal CC}(M)$ hence ${\cal P}_{nc}(M)$ is pathwise connected. Since $\frac{2\pi}{\eta}\alpha_j$ has no closed leaf, $\frac{2\pi}{\eta}\alpha_j\in{\cal P}_{nc}(M)$, therefore there is a path $\alpha:[0,1]\rightarrow {\cal P}(M)$ such that $\alpha(0)=\frac{2\pi}{\eta}\alpha_1$ and that $\alpha(1)=\frac{2\pi}{\eta}\alpha_2$. Since ${\cal D}(M)$ is open, we can change $\alpha$ so that we have $\alpha(t)\in{\cal I}\cap {\cal D}(M)$ for any $t\in [0,1]$ (cf. \cite{notes}). Thus $\alpha(t)$ is an arational lamination (up to cutting $\partial M$ along a closed leave of $\alpha(t)$ if there is one) lying in ${\cal D}(M)$. From Lemma \ref{arra} we get  $\alpha(t)\in\hat {\cal O}$ for any $t\in [0,1]$.\\
\indent Let $\kappa_j:[0,1]\rightarrow \hat{\cal O}$ be the path $\kappa_j(t)=(1-t+t\frac{2\pi}{\eta})\alpha_j$. The union of the paths $k_j$, $\kappa_j$ for $j=1,2$ and of the path $\alpha([0,1])$ is a path lying in $\hat {\cal O}$ joining $\lambda_1$ to $\lambda_2$.\\
\indent
We have proved that $\hat{\cal O}$ is pathwise connected. Taking  $\lambda_1,\lambda_2\in {\cal D}(M)$ at the beginning of this proof, we get that ${\cal D}(M)$ is also pathwise connected.
\end{proof}

\section{Pleated surfaces}
\begin{theorem}
Let $M$ be an orientable 3-manifold, let $\rho:\pi_1(M)\rightarrow Isom(\Hp^3)$ be a geometrically finite representation uniformizing $N$ and having only rank $2$ maximal parabolic subgroups and let $h:N=\Hp^3/\rho(\pi_1(M))\rightarrow int(M)$ be a homeomorphism; then any measured geodesic lamination $\lambda\in{\cal D}(M)$ is realized by a pleated surface in $N$.
\end{theorem}

\begin{proof}
If $M$ is a compression body and $\lambda$ is arational, then $\lambda$ lies in the Masur domain and the theorem has been proved by Otal (\cite{chefjeune}). If $M$ is boundary irreducible, then any geodesic lamination is realized in $N$ (see \cite[chap. 5]{ceg}). In order to prove our general statement, we will follow the main lines of Otal's proof.

\begin{lemma}
Let $\lambda\in{\cal D}(M)$ be a weighted multi-curve, then $\lambda$ is realized by a pleated surface in $N$.
\end{lemma}
\begin{proof}
Let us extend $|\lambda|$ to  a geodesic lamination $L$ (namely $|\lambda|\subset L$) such that all the components of $\partial M-L$ are  triangles and that $L$ has finitely many leaves. Since $\lambda\in {\cal D}(M)$ and since $\rho$ has only rank $2$ cusps, any closed leaf of $L$ is homotopic  to a closed geodesic in $N$. Let $S\subset M$ be a properly embedded surface homeomorphic and homotopic to $\partial M$ and let us change the restriction of $h$ to $S$ by a homotopy  in order to get a map $f:S\rightarrow N$ mapping the closed leaves of $L$ into closed geodesics. For each connected component of $S$, let us lift this to a map  $\hat f:\Hp^2\rightarrow\Hp^3$; this map $\hat f$ defines a map from the endpoints of the lifts of the leaves of $L$ to $L_\rho$. Furthermore, if $\hat l\in\Hp^2$ is a lift of a leaf of $L$, by Lemma \ref{supcond}, the images of its two endpoints are distincts. Following \cite[Theorem 5.3.6]{ceg}, this allows us to construct a pleated surface realizing $L$. 
\end{proof}

\indent
Now let us consider the general case. Let $\lambda\in{\cal D}(M)$ be a measured geodesic lamination; let $\lambda_n$ be  a sequence of weighted multi-curves such that $\lambda_n\longrightarrow\lambda$ in ${\cal ML}(\partial M)$ and that $|\lambda_n|\rightarrow |\lambda|$ in the Hausdorff topology. Since ${\cal D}(M)$ is open, $\lambda_n\in{\cal D}(M)$ for large $n$. Let $\gamma$ be a weighted multi-curve with a maximal number of leaves such that $i(\lambda,\gamma)=0$; since $\lambda_n\in{\cal D}(M)$  for large $n$, $\lambda_n\cup\gamma$ is also a measured geodesic lamination lying in ${\cal D}(M)$. By the previous lemma, $\lambda_n\cup\gamma$ is realized by a pleated surface  $f_n: S\rightarrow N$. We will show that a subsequence of $(f_n)$ converges to a pleated surfaces realizing $\lambda$.\\
\indent
 Let  us denote by $s_n$ the metric on $S$ induced by the map $f_n:S\rightarrow N$ and let us show that $(s_n)$ contains a converging subsequence. First we will prove that the sequence of metrics $(s_n)$ is bounded in the modular space. By Mumford's Lemma, it is sufficient to prove that the injectivity radius of $s_n$ is bounded from below.

\begin{claim}	\label{ptitcourb}
Let $(c_n)$ be a sequence of curves such that $l_{s_n}(c_n)\longrightarrow 0$ and let us extract a subsequence $(c_n)$ which converges in the Hausdorff topology to a geodesic lamination $C$; then $C$ does not intersects $\lambda$ transversely.
\end{claim}

\begin{proof}
Let assume the contrary and let $c$ be a leaf of $C$ intersecting $\lambda$ transversely. Since $\lambda$ is recurrent, we can consider a segment $k=k([0,1])$ of  $|\lambda|$ such that $k\cap C=\partial k$ and that $\frac{dk}{dt}(0)$ is close (for some reference metric on $S$) to $-\frac{dk}{dt}(1)$ and a short segment $\kappa$  of $c$ joining the ends of $k$ so that we get a closed curve  $d=k\cup\kappa$. Since $\lambda_n\longrightarrow\lambda$ and $ c_n\longrightarrow C$, there exists arcs $k_n\subset\lambda_n$ and $\kappa_n\subset c_n$ near $k$ and $\kappa$ such that $d_n=k_n\cup\kappa_n$ is homotopic on $S$ to $d$. Since $l_{s_n}( c_n)\longrightarrow 0$, $ c_n$ is the core of a very deep Margulis tube and $l_{s_n}(k_n)\longrightarrow\infty$. Since $l_{s_n}(\kappa_n)\leq l_{s_n}( c_n)\longrightarrow 0$ and $f_n(k_n)\subset f_n(\lambda_n)$ is a geodesic arc, $f_n(d_n)=f_n(k_n\cup\kappa_n)$ is a quasi-geodesic and is very close to the geodesic $d^*_n$ of $N$ in its homotopy class. This implies that $l_{\rho}(d^*_n)\longrightarrow\infty$ but $d_n$ is homotopic to $d$ so $d_n^*=d^*$ giving the expected contradiction.
\end{proof}

Let $(c_n)$ be a sequence of curves such that $l_{s_n}(c_n)\longrightarrow 0$. If we can extract a converging (in the Hausdorff topology) subsequence such that all the $c_n$ are meridians then, by Casson's criterion (cf. \cite{chefjeune}, \cite[Theorem B.1]{espoir}), the limit contains a homoclinic leaf. By Lemma \ref{supcond} such a homoclinic leaf intersects $\lambda$ transversely contradicting Claim \ref{ptitcourb}. This implies that for large $n$, the $c_n$ are not meridians. If we can extract a converging subsequence such that all the $c_n$ are parabolic curves, then $i(c_n,\lambda)>\eta$ for any $n$, leading to the same contradiction.\\
\indent It follows that, for large $n$, each $f_n(c_n)$ is homotopic to a closed geodesic $ c_n^*$ of $N$. But this would mean that $l_{\rho}( c_n^*)\longrightarrow 0$ and since $N$ is geometrically finite, there is a uniform lower bound for the length of a closed geodesic.  We get then from Mumford's Lemma (\cite[Proposition 3.2.13]{ceg}) :
\begin{claim}
The sequence $(s_n)$ is bounded in the moduli space.
\end{claim}

Let us now show that $(s_n)$ is bounded in the Teichm\"uller space. By the previous claim, there exists a sequence $(\varphi_n)$ of diffeomorphisms such that, up to extracting a subsequence, $(\varphi_n^* s_n)$ converges in the Teichm\"uller space to a metric $s'_{\infty}$. By construction $l_{\varphi_n^* s_n} (\varphi_n^{-1}(\gamma))=l_{s_n}(\gamma)=l_{\rho}(\gamma)$, therefore the $s'_{\infty}$-length of the multi-curve $\varphi^{-1}_n (\gamma)$ is bounded. This implies that we can choose some $n_0$ and a subsequence such that any diffeomorphism $(\varphi_n^{-1}\circ\varphi_{n_0})$ preserves this multi-curve, component by component.\\
\indent For large $n$,  $\lambda_n$ intersects transversely all the parabolic curves. Therefore $\lambda_n$ lies in the thick part of $N$ which is compact. It follows that all the $f_n(S)$ intersect the same compact subset of $N$. Using Ascoli's theorem we can choose a subsequence of $(\varphi_n)$ such that the sequence of pleated surfaces $(f_n\circ\varphi_n)$ converges. This implies that the maps $f_n\circ\varphi_n$ are homotopic for $n$ sufficiently large. Thus, up to changing $n_0$, the diffeomorphisms $\psi_n=\varphi_n^{-1}\circ\varphi_{n_0}$ are homotopic in $M$ to the identity. Let $R$ be a complementary region of $\gamma$. If the map $i^*:\pi_1(R)\rightarrow\pi_1(M)$ induced by the inclusion is injective, then by \cite{bois}, $\psi_{n|R}$ is isotopic to the identity in $S$. If the map $i^*:\pi_1(R)\rightarrow\pi_1(M)$ is not injective, $R$ contains a meridian. Since $\lambda\in{\cal D}(M)$, $R$ must contain a component $\lambda^i$ of $\lambda$ and since $\gamma$ has a maximal number of components, $\lambda^i$ must be arational in $R$. Let us call $r_n$ the restriction of $s_n$ to $R$ and suppose that the sequence $(r_n)$ is not bounded in Teichm\"uller space. Since the length of $\partial R$ is bounded, we can use Thurston's compactification and assume that $(r_n)$ tends to a measured geodesic lamination $\nu$. Since $l_{r_n}(\lambda_n\cap R)=l_{\rho}(\lambda_n\cap R)\leq l_{r_{n_0}}(\lambda_n\cap R)\rightarrow l_{r_{n_0}}(\lambda^i)$, $i(\nu,\lambda^i)=0$ and $\nu$ and $\lambda^i$ share the same support.\\
\indent
 Let $m\subset R$ be a meridian. Then $m_n=\psi_n(m)$ is homotopic to $m$ and therefore $(m_n)$ is a sequence of meridians. We can assume that $(m_n)$ converges in ${\cal PML}$ to a projective measured lamination represented by $\mu$. Since $(\psi_n^*s_n)$ converges, then $l_{s_n}(m_n)=l_{\psi_n^*s_n}(\psi_n^{-1}(m_n))=l_{\psi_n^*s_n}(m)$ converges and therefore $i(\mu,\nu)=0$. Since $\nu$ and $\lambda^i$ have the same support and since $\lambda^i$ is arational in $R$, this implies that $\mu$ and $\lambda^i$ have the same support. But the Casson's criterion (c.f. \cite{chefjeune}, \cite[Theorem B.1]{espoir}) says that there exists a simple geodesic $l\subset R$ which is homoclinic and does not intersect $\mu$ transversely. This contradicts Lemma \ref{supcond} and proves that the sequence $(r_n)$ is bounded.\\
\indent
This applies to each component of $\partial M-\gamma$. It follows that we can choose the $\psi_n$ such that each one is the composition of Dehn twists along the leaves of $\gamma$. We have seen above that the $\psi_n$ are homotopic to the identity; by \cite{bois}, each $\psi_n$ can be extended to a homeomorphism of the whole manifold $M$. Let ${\cal V}\subset S$ be a small neighbourhood of $\gamma$; since $\lambda\subset{\cal D}(M)$, ${\cal V}$ does not contain the boundary of any essential annulus. It follows then from \cite[Proposition 27.1]{jojo} that, up to isotopy, each $\psi_n$ has finite order. Since the $\psi_n$ are compositions of Dehn twists along disjoint curves, they can not have finite order except when they are isotopic to the identity.  We get from \cite{ceg}  that a subsequence of $(f_n)$ converges to a pleated surface realizing $\lambda$.
\end{proof}

Let $f:S\rightarrow N$ be a pleated surface realizing a geodesic lamination $L$. Let $\Pe (N)$ be the tangent line bundle of $N$. We define a map $\Pe  f$ from $L$ to  $\Pe (N)$ by mapping a point $x\in L$ to the direction of the unit vector tangent to $f(L)$ at $f(x)$.\\
\indent
The following injectivity theorem has been proved by Thurston (\cite{thui}) when $M$ is boundary irreducible and by Otal (\cite{chefjeune}) when $M$ is a compression body and $\lambda\in\hat{\cal O}$.

\begin{theorem}
Let $\lambda\in{\cal D}(M)$ be a measured geodesic lamination not satisfying the condition $(-)$, let $L$ be a geodesic lamination containing the support of $\lambda$ and let\linebreak $f:\partial M\rightarrow N$ be a pleated surface realizing $L$. Then the map from $\Pe f:L\rightarrow \Pe (N)$ is a homeomorphism into its image.
\end{theorem}

\begin{proof}
Since the map $f$ reduces the length, it is easy to see that $\Pe f$ is a continuous map and since $L$ is compact, we need only to show that $\Pe f$ is injective.\\
\indent
Let us assume the contrary, there are two points $u$ and $v\subset L$ such that\linebreak $\Pe f(u)=\Pe f(v)$; let $\hat f:\Hp^2\rightarrow\Hp^3$ be a lift of $f$ and let $\hat u$ and $\hat v$ be lifts of $u$ and $v$ such that $\Pe\hat  f(\hat u)=\Pe\hat f(\hat v)$. Since $\hat f$ is an isometry on the preimage of $L$, it is injective on each leaf of the preimage of $L$. Therefore $\hat u$ and $\hat v$ lie in two different leaves $\hat l_1$ and $\hat l_2$ of the preimage of $L$. Since $\Pe\hat f(\hat u)=\Pe\hat f(\hat v)$, then $\hat f(\hat l_1)=\hat f(\hat l_2)$. It follows that $L$ is an annular lamination and since $L$ does not intersect $\lambda\in{\cal D}(M)$ transversely, this contradicts Lemma \ref{supcond}.
\end{proof}

\begin{rem}
If $\lambda$ satisfies the condition $(-)$, the same is true for $\lambda$ but not for any geodesic lamination containing $\lambda$.
\end{rem}

\section{Action on $\R$-trees}
\indent
We will prove the following :

\begin{proposition}	\label{reali}
Let ${\cal T}$ be a real tree, let $\pi_1(M)\times {\cal T}\rightarrow{\cal T}$ be a small minimal action and let $\lambda\in{\cal D}(M)$ be a measured geodesic geodesic lamination. Then at least one connected component of $\lambda$ is realized in ${\cal T}$.
\end{proposition}

\begin{proof}
Let us first notice that this result has been proved by G. Kleineidam and J. Souto (\cite{boches} and \cite{boches1}) when $M$ is a compression body and $\lambda$ lies in the Masur domain. The general case need just a reorganization of the proof of \cite[Proposition 6]{espoir}. Here we will sketch the proof  which consists essentially in putting together ideas of \cite{meister1} and of \cite{boches}.\\
\indent
If $\lambda$ satisfies the condition $(-)$ then the elements of $\pi_1(M)$ corresponding to the leaves of $\lambda$ form a generating subset of $\pi_1(M)$. In this case Proposition \ref{reali} is a straightforward consequence of \cite{mors1}.\\
\indent
Let us assume that $\lambda$ does not satisfies the condition $(-)$. For $c\in\pi_1(M)$ let us denote by $\delta_{\cal T}(c)$ the distance of translation of $c$ on ${\cal T}$.
Let $S$ be a connected component of $\partial M$ with $\chi(S)<0$; the inclusion $i_*:\pi_1(S)\rightarrow\pi_1(M)$ provides us with an action of $\pi_1(S)$ on ${\cal T}$. By \cite{chefm}, there exists a measured geodesic lamination $\beta\in{\cal ML}(S)$ and a morphism $\phi:{\cal T}_{\beta}\rightarrow {\cal T}_S$ from the dual tree of $\beta$ to the minimal subtree of ${\cal T}$ that is invariant by the action of $\pi_1(S)$. Since the action of $\pi_1(S)$ is not a priori small, $\phi$ is not, a priori, an isomorphism and there might be many laminations $\beta$ with this property. We will consider such a lamination $\beta$ which is adapted to our problem.\\
\indent
Let $(\lambda_n)$ be a sequence of weighted multi-curves converging to $\lambda$ in ${\cal ML}(\partial M)$ such that $(|\lambda_n|)$ converges to $|\lambda|$ in the Hausdorff topology. For each irrational  sublamination $\lambda^i$ of $\lambda$  let us denote by $S(\lambda^i)$ the surface embraced by $|\lambda^i|$. For $n$ large enough such that $|\lambda_n|$ does not intersect $\partial'\bar S(\lambda)$ transversely, let us add simple closed curves to $\partial'\bar S(\lambda)\cup|\lambda_n|$ in order to obtain a multi-curve $L_n$ whose complementary regions are pairs of pants. By $\cite{chefm}$, there are measured geodesic laminations $\beta_n\in{\cal ML}(\partial M)$ and equivariant morphisms $\phi_n:{\cal T}_{\beta_n}\rightarrow {\cal T}$ such that for any leaf $l_n$ of $L_n$, either $\delta_{\cal T}(l_n)>0$ and the restriction of $\phi_n$ to the axis of $l_n$  is an isometry or $\delta_{\cal T}(l_n)=0$ and $i(l_n,\beta_n)=0$, see \cite[\S 4.1]{espoir} for more details.\\
\indent
Extract a subsequence such that $(|\beta_n|)$ converges to a geodesic lamination $B$ in the Hausdorff topology. The first step of the proof is to show that $B$ intersects $|\lambda|$ transversely, this will allow us to follow \cite{boches} by using a realization of a train track carrying $\lambda$ to prove the proposition.

\begin{lemma}	\label{trans}
The geodesic lamination $B$ intersects $|\lambda|$ transversely.
\end{lemma}

\begin{proof}
The proof is done by contradiction; let us assume that $|\lambda|$ does not intersect $B$ transversely.\\
\indent
If $B$ is a multi-curve, then for large $n$, $|\beta_n|=B$ and $\beta_n$ does not intersect $\lambda$ transversely. By the definition of ${\cal D}(M)$, a small neighbourhood of $B$ does not contain any essential disk, annulus or Moebius band. By \cite[Corollary IV 1.3]{mors1}, this implies that the action of $\pi_1(M)$ fixes a point of ${\cal T}$. This would contradict the assumption that this action is minimal.\\
\indent
Let us now consider the case where $B$ is not a multi-curve. The first step in this case is to prove that $S(B)$ is incompressible for any connected component $B^i$ of $B$. This will implies that a subsequence of $(|\beta_n|)$ is constant.

\begin{claim}	\label{uncon}
If $B$ does not intersects $|\lambda|$ transversely, then for any connected component $B^i$ of $B$, the surface $S(B^i)$ is incompressible.
\end{claim}

\begin{proof}
Since we have assumed that $B$ does not intersect $|\lambda|$ transversely, if $B^i$ is a closed curve, the claim follows from the definition of ${\cal D}(M)$.\\
\indent
Let $B^i$ be a component of $B$ which is not a closed curve and let us assume that $S(B^i)$ contains a meridian. It follows from the ideas of \cite{boches}, that $S(B^i)$ contains a homoclinic leaf $h$ which does not intersect $B^i$ transversely (see \cite[Lemma 4.3]{espoir} for details). Since we have assumed that $B$ does not intersect $\lambda$ transversely, then $|\lambda|\cap S(B^i)\subset B^i$. Especially, $h$ does not intersect $\lambda$ transversely, contradicting Lemma \ref{supcond}.
\end{proof}

\indent
Let us explain how Claim \ref{uncon} implies that for large $n$ the support of $\beta_n$ does not depend on $n$. Let $B^i$ be a connected component of $B$; if $B^i$ is a closed leaf then for large $n$, $B^i\subset|\beta_n|$. Let us next assume that $B^i$ is not a closed leaf; by claim \ref{uncon}, $S(B^i)$ is incompressible, hence the action of $i_*(\pi_1(S(B^i))$ on its minimal subtree ${\cal T}_{S(B^i)}\subset{\cal T}$ is small. Since $B$ does not intersect $\partial'\bar S(B^i)$, for large $n$, $\beta_n$ does not intersect $\partial'\bar S(B^i)$. It follows that for each component $d$ of $\partial'\bar S(B^i)$, the action of $i_*(d)$ has a fixed point in ${\cal T}_ {S(B^i)}$. This allows us to apply Skora's theorem \cite{skora} which says that $\beta_n^i=\beta_n\cap S(B^i)$ is dual to the action of $i_*(\pi_1(S(B^i))$ on ${\cal T}_{S(B^i)}$. Doing this for each component of $B$, we obtain that, for large $n$, $|\beta_n|$ does not depend on $n$. Let us endow $B$ with the measure of one of the $\beta_n$ and let us call $\beta$ the measured geodesic lamination thus obtained.\\

\indent
The last step in the proof of Lemma \ref{trans} is to show that $|\beta|=B$ is annular. Since we have assumed that $B$ does not intersect $|\lambda|$ transversely, this will contradict the fact that $\lambda\in{\cal D}(M)$ (Lemma \ref{supcond}).

\begin{claim}
The measured geodesic lamination $\beta$ is annular
\end{claim}

\begin{proof}
By hypothesis $\beta$ does not intersect $\lambda$ transversely hence $S(\beta)\cap|\lambda|\subset|\beta|$.\\
\indent
Since $S(\beta)$ is incompressible, we might consider a characteristic submanifold $W$ of $(M,S(\beta))$ (cf. \cite{jojo} and \cite{jacs}). Such a characteristic submanifold is a union of essential $I$-bundles  and Seifert fibered manifolds such that any essential annulus in $(M,S(\beta))$ can be homotoped in $W$. For each component $\Sigma$ of $\partial M-S(B)$, $i_*(\Sigma)$ fixes a point in ${\cal T}$, hence by \cite{thuiii} (see also \cite[theorem IV 1.2]{mors2}) $W$ can be isotoped in such a way that we have $\beta\subset W\cap\partial M$.\\
\indent
We are considering the case where $\beta$ is not a multi-curve, therefore it contains an irrational sublamination $\beta^1$. Since the Seifert fibered manifolds composing $W$ intersect $\partial M$ in annuli, $|\beta^1|$ lies in a component $W^1$ of $W$ which is an essential $I$-bundle over a compact surface $F$ : $W^1=F\times I$. Let us denote by $p:F\times\partial I\rightarrow F$ the projection along the fibers. By Skora's theorem \cite{skora}, for any component $\Sigma$ of $W^1\cap\partial M$, $\Sigma\cap\beta$ is dual to the action of $i_*(\pi_1(S))$ on ${\cal T}_{\Sigma}$. Since this action factorizes through the action of $\pi_1(W^1)=\pi_1(F)$, there is a measured geodesic lamination $\beta'\in{\cal ML}(F)$ such that $\beta\cap\partial W^1\supset p^{-1}(\beta')$. Since the lamination $p^{-1}(\beta')$ is annular, $\beta$ is annular (compare with \cite[Lemma 14]{meister1}).
\end{proof}

\indent
This claim concludes the proof of Lemma \ref{trans}. 
\end{proof}

\indent
Let us now complete the proof of Proposition \ref{reali}. Let $\lambda^i$ be a connected component of $\lambda$ that intersects $B$ transversely. Let us denote by $\pi_{\beta_n}:\Hp^2\rightarrow{\cal T}_{\beta_n}$ the projection associated to the dual tree of $\beta_n$ (as defined in \S \ref{def}). Since $B$ intersects $\lambda^i$ transversely, the construction in \cite[chap 3]{chefjeune} yields a train track $\tau^i$ such that for large $n$, $\pi_{\beta_n}$  is a weak realization of $\tau^i$ in ${\cal T}_{\beta_n}$.\\
\indent
Let $l_n$ be a component of $L_n\cap S(\lambda^i)$. Up to extracting a subsequence, $l_n$ converge in the Hausdorff topology to a geodesic lamination $L'\subset S(\lambda^i)$ that does not intersect $\lambda^i$ transversely (by the choice of $L_n$). Therefore $|\lambda^i|\subset L'$. If up to extracting a subsequence, $i_*(l_n)$  has a fixed point in ${\cal T}$; then $i(\beta_n,l_n)=0$. Letting $n$ tends to $\infty$, we would get that $B$ does not intersect $|\lambda^i|$ transversely, contradicting our choice of $\lambda^i$.\\
\indent
It follows from the previous paragraph that the restriction of $\phi_n$ to $l_n$ is an isometry. For large $n$, each branch of $\hat\tau$ intersects transversely a lift of $l_n$. The fact that the restriction of $\phi_n$ to the axis of $l_n$ is an isometry implies that $\phi_n\circ\pi_{\beta_n}$ is a weak realization of $\tau^i$ in ${\cal T}$ (compare with \cite[Lemma 11]{boches}). By \cite{chefjeune} this map $\phi_n\circ\pi_{\beta_n}$ is homotopic to a realization of $\lambda^i$ in ${\cal T}$.
\end{proof}

\indent
Let $\rho_n:\pi_1(M)\rightarrow Isom(\Hp^3)$ be a sequence of representations containing no converging subsequence; in \cite{mors1}, J. Morgan and P. Shalen described a way to associate a small minimal action of $\pi_1(M)$ on an $\R$-tree to some subsequence of $(\rho_n)$. This can be stated in the following way : the sequence $(\rho_n)$ tends to the action $\pi_1(M)\curvearrowright{\cal T}$ in the sense of Morgan and Shalen if there is a sequence $\varepsilon_n\longrightarrow 0$ such that for any $a\in\pi_1(M)$, $\varepsilon_n\delta_{\rho_n}(a)\longrightarrow \delta_{\tau}(a)$.
In \cite{conti}, J.-P. Otal described, in the special case of handlebodies, the behavior of the length of measured geodesic laminations which are realized in ${\cal T}$. A careful look at the proof yields the following statement.

\begin{theorem}[Continuity Theorem \cite{conti}]	\label{contin}
Let $(\rho_n)$ be a sequence of discrete and faithful representations of $\pi_1(M)$ tending in the sense of Morgan and Shalen to a small minimal action of $\pi_1(M)$ on an $\R$-tree ${\cal T}$. Let $\varepsilon_n\longrightarrow 0$ be such that $\forall g\in\pi_1(M)$, $\varepsilon_n \delta_{\rho_n}(g)\longrightarrow \delta_{\cal T}(g)$ and let $L\subset\partial M$ be a geodesic lamination which is realized in ${\cal T}$. Then there exists a neighbourhood ${\cal V}(L )$ of $L$, and constants $K,n_0$ such that for any simple closed curve $c\subset {\cal V}(L )$ and for any $n\geq n_0$,
$$\varepsilon_n l_{\rho_n}(c^*)\geq K  l_{s_0}(c).$$
\end{theorem}

In the preceding statement $s_0$ is a fixed complete hyperbolic metric on $\partial_{\chi<0} M$. Using this and Proposition \ref{reali}, we get the following

\begin{theorem}
Let $\rho_n$ be a sequence of faithful representations of $\pi_1(M)$ such that $\Hp^3/\rho_n(\pi_1(M))$ is homeomorphic to $int(M)$, let $\lambda\in{\cal D}(M)$ and let $\lambda_n$ be a sequence of measured geodesic laminations such that :
\begin{description}
\item - the sequence $\lambda_n$ converges to $\lambda$ in ${\cal ML}(\partial M)$;
\item - the sequence $|\lambda_n|$ converges to $|\lambda|$ in the Hausdorff topology;
\item - the sequence $l_{\rho_n}(\lambda_n)$ is bounded.
\end{description} 
Then $(\rho_n)$ contains a converging subsequence.
\end{theorem}

\begin{proof}
Approximating each $\lambda_n$ by weighted multi-curves, we produce a sequence of multi-curves also satisfying the hypothesis of the theorem. Let us assume that $(\rho_n)$ does not contain an algebraically converging subsequence, then by \cite{mors1}, a subsequence of $(\rho_n)$ tends to a small minimal action of $\pi_1(M)$ on an $\R$-tree ${\cal T}$.  By Proposition \ref{reali}, $\lambda$ is realized in ${\cal T}$ and it follows from Theorem \ref{contin} that $l_{\rho_n}(\gamma_n)\longrightarrow\infty$ giving us the desired contradiction.
\end{proof}

\begin{rem}
When $M$ is an $I$-bundle over a closed surface, the proof of this theorem can be found in \cite{thui}; this result has been extended to manifolds with incompressible boundary in \cite{ohsh2}. When $M$ is a compression body and $\lambda\in\hat{\cal O}$, this result has been proved in \cite{boches} and \cite{boches1}.
\end{rem}

\section{Conclusion}

To complete this paper, we should also mention the action of $Mod(M)$ on ${\cal D}(M)$. The following result is proved in \cite{fini} using some properness properties of the bending map. The proof of these properties is long and is subject of \cite{proper}. Here we will only give an outline of the proof, the reader interested in a complete proof should refer to \cite{fini} or to \cite{proper}.  

\begin{proposition}	\label{un}
If $M$ is not a genus $2$ handlebody, the action of $Mod(M)$ on ${\cal D}(M)$ is properly discontinuous.
\end{proposition}

\begin{outline}
Here $Mod(M)$ is the group of isotopy classes of diffeomorphisms $M\rightarrow M$.\\
\indent
Let us assume that Proposition \ref{un} is not true. There are measured geodesic laminations $\lambda\in{\cal D}(M)$, $(\lambda_n)\in {\cal D}(M)$ and diffeomorphisms $(\phi_n)\in Mod(M)$ such that $(\lambda_n)$ and  $(\phi_n(\lambda_n))$ converge to $\lambda$ in ${\cal ML}(\partial M)$ and that for any $n\neq m$, $\phi_n$ is not isotopic to $\phi_m$. Since $\lambda\in{\cal D}(M)$, $\exists \eta>0$ such that $i(\lambda,\partial D)>\eta$ for any essential disc $D$. Let $\frac{2\pi}{\eta} \lambda$ be the measured geodesic lamination obtained by rescaling the measure of  $\lambda$ by $\frac{2\pi}{\eta}$. Let $\lambda^i$ be a compact leaf of $\frac{2\pi}{\eta} \lambda$ with a weight greater than or equal to $\pi$; if, up to extracting a subsequence, $\lambda^i$ is a compact leaf of all the measured geodesic laminations $\lambda_n$, let us replace, in $\frac{2\pi}{\eta} \lambda$ and in all $\frac{2\pi}{\eta} \lambda_n$,  $\lambda^i$ by a the same leaf with weight $\pi$. Let $\lambda'_\infty$ and $\lambda'_n$ be the measured geodesic laminations obtained by doing the same for all the leaves of $\frac{2\pi}{\eta} \lambda$ with a weight greater than $\pi$; let us remark that $\lambda'_\infty$ may have some leaves with a weight greater than $\pi$ but that for $n$ large enough, the compact leaves of $\lambda'_n$ have a weight less than or equal to $\pi$. Let us also remark that $(\lambda'_n)$ and $(\phi_n(\lambda'_n))$ converge to $\lambda'_\infty$ in ${\cal ML}(\partial M)$. By Lemma \ref{dedan}, $\lambda'_\infty$ and $\lambda'_n$ satisfy the conditions $b)$, $c)$. For $n$ large enough, the $\lambda'_n$ also satisfy the condition $a)$ hence,  by \cite{espoir} (see also \cite{fini}), there is a geometrically finite metric $\rho_n$ on the interior of $M$ whose bending measured lamination is $(\lambda'_n)$; here a geometrically finite metric is a geometrically finite representation $\rho:\pi_1(M)\rightarrow Isom(\Hp^3)$ together with an isotopy class of homeomorphisms $M\rightarrow N^{ep}$. The bending measured geodesic lamination of $\phi_{n*}(\rho_n)$ is $\phi_n(\lambda'_n)$ and by construction $\phi_n(\lambda'_n)\longrightarrow\lambda'_\infty$. It is at this point that we need the properness property of the bending map mentioned before the statement of Proposition \ref{un} : it follows from \cite{espoir} that  there  is a subsequence such that $(\rho_n)$ and $(\phi_{n*}(\rho_n))$ converge to some geometrically finite metrics.\\
\indent
The conclusion comes from the fact that the action of $Mod(M)$ on the space of isotopy classes of geometrically finite metrics (see \cite{proper} for a definition) on the interior of $M$ is properly discontinuous. This fact can be shown by using the arguments of the proof of the properness properties mentioned above (cf. \cite{proper}).
\end{outline}

\indent
As has been mentioned throughout this paper, almost all the above results have been already proved when $\lambda\in\hat{\cal O}$. In an attempt to convince the reader of the interest of this paper we will give some examples of laminations lying in ${\cal D}$ but not in $\hat{\cal O}$.\\
\indent
Let $M$ be an $I$-bundle over a compact surface $S$ with boundary; this manifold $M$ is a handlebody. Let $(\gamma,\alpha)\in{\cal ML}(S)$ be a pair of binding measured geodesic laminations, namely for any measured geodesic lamination $\beta\in {\cal ML}(S)$,\linebreak $i(\beta,\gamma)+i(\beta,\alpha)>0$. Such a pair  of binding measured geodesic laminations has the following property : $\exists\eta>0$ such that $i(c,\gamma)+i(c,\alpha)\geq\eta$ for any closed curve $c\subset S$. Let us defined a measured geodesic lamination $\lambda\in{\cal ML}(\partial M)$ as follows : on one component $\{0\}\times S$ of $\partial I\times S$, $\lambda\cap (\{0\}\times S)$ is $\gamma$, on the other component, $\lambda\cap (\{1\}\times S)$ is $\alpha$ and on the remaining part $I\times \partial S$ of the boundary, $\lambda\cap (I\times \partial S)$ is $\{p\}\times\partial S$ for some $p\in ]0,1[$ endowed with a Dirac mass $\eta$.\\
\indent
For any essential disc $D\subset M$, $\partial D$ intersects $\{p\}\times\partial S$, hence $i(\partial D,\lambda)\geq\eta$. If $A$ is an essential annulus, either $\partial A$ intersects $\{p\}\times\partial S$ and $i(\partial A,\lambda)\geq\eta$, or $A$ can be homotoped to a vertical annulus $c\times I\subset I\times S$ with $c$ being a simple closed curve. In the second case, we have $i(\partial A,\lambda)=i(c,\gamma)+i(c,\alpha)\geq\eta$. We have thus proved that $\lambda\in{\cal D}(M)$. By \cite{boches} the measured geodesic laminations $\lambda\cap\{0\}\times S$ and  $\lambda\cap\{1\}\times S$ have the same supports as some measured laminations lying in ${\cal M}'$ hence $\lambda\not\in\hat{\cal O}$.

\end{document}